\documentclass[11pt]{amsart} 
\usepackage{amssymb,amsthm,amsmath,epsfig,latexsym}
\usepackage{graphicx}
\usepackage{calc,times}
\usepackage{geometry} 
\begin{document}

\newcommand{\mmbox}[1]{\mbox{${#1}$}}
\newcommand{\proj}[1]{\mmbox{{\mathbb P}^{#1}}}
\newcommand{\Cr}{C^r(\Delta)}
\newcommand{\CR}{C^r(\hat\Delta)}
\newcommand{\affine}[1]{\mmbox{{\mathbb A}^{#1}}}
\newcommand{\Ann}[1]{\mmbox{{\rm Ann}({#1})}}
\newcommand{\caps}[3]{\mmbox{{#1}_{#2} \cap \ldots \cap {#1}_{#3}}}
\newcommand{\N}{{\mathbb N}}
\newcommand{\Z}{{\mathbb Z}}
\newcommand{\R}{{\mathbb R}}
\newcommand{\sat}{{\rm sat}}
\newcommand{\Tor}{\mathop{\rm Tor}\nolimits}
\newcommand{\Ext}{\mathop{\rm Ext}\nolimits}
\newcommand{\Hom}{\mathop{\rm Hom}\nolimits}
\newcommand{\im}{\mathop{\rm Im}\nolimits}
\newcommand{\rank}{\mathop{\rm rank}\nolimits}
\newcommand{\supp}{\mathop{\rm supp}\nolimits}
\newcommand{\arrow}[1]{\stackrel{#1}{\longrightarrow}}
\newcommand{\CB}{Cayley-Bacharach}
\newcommand{\coker}{\mathop{\rm coker}\nolimits}
\sloppy
\newtheorem{defn0}{Definition}[section]
\newtheorem{prop0}[defn0]{Proposition}
\newtheorem{quest0}[defn0]{Question}
\newtheorem{thm0}[defn0]{Theorem}
\newtheorem{lem0}[defn0]{Lemma}
\newtheorem{corollary0}[defn0]{Corollary}
\newtheorem{example0}[defn0]{Example}
\newtheorem{remark0}[defn0]{Remark}

\newenvironment{defn}{\begin{defn0}}{\end{defn0}}
\newenvironment{prop}{\begin{prop0}}{\end{prop0}}
\newenvironment{quest}{\begin{quest0}}{\end{quest0}}
\newenvironment{thm}{\begin{thm0}}{\end{thm0}}
\newenvironment{lem}{\begin{lem0}}{\end{lem0}}
\newenvironment{cor}{\begin{corollary0}}{\end{corollary0}}
\newenvironment{exm}{\begin{example0}\rm}{\end{example0}}
\newenvironment{rem}{\begin{remark0}\rm}{\end{remark0}}

\newcommand{\defref}[1]{Definition~\ref{#1}}
\newcommand{\propref}[1]{Proposition~\ref{#1}}
\newcommand{\thmref}[1]{Theorem~\ref{#1}}
\newcommand{\lemref}[1]{Lemma~\ref{#1}}
\newcommand{\corref}[1]{Corollary~\ref{#1}}
\newcommand{\exref}[1]{Example~\ref{#1}}
\newcommand{\secref}[1]{Section~\ref{#1}}
\newcommand{\remref}[1]{Remark~\ref{#1}}
\newcommand{\questref}[1]{Question~\ref{#1}}

\newcommand{\std}{Gr\"{o}bner}
\newcommand{\jq}{J_{Q}}


\title{Error-correction of linear codes via colon ideals}

\author{Benjamin Anzis and \c{S}tefan O. Toh\v{a}neanu}

 \keywords{linear codes, minimum distance, saturation, colon ideals, MDS codes, Castelnuovo-Mumford regularity, free resolutions.\\ \indent The second author is the corresponding author.\\
\indent Authors' addresses: Department of Mathematics, University of Idaho, Moscow, ID 83844, anzi4123@vandals.uidaho.edu, tohaneanu@uidaho.edu.}

\begin{abstract}
\noindent We show that errors in data transmitted through linear codes can be
thought of as codewords of minimum weight of new linear codes. To determine
errors we can then use methods specific to finding such special codewords. One
of these methods consists of finding the primary decomposition of the
saturation of a certain homogeneous ideal. When good words (i.e.\ vectors with
a unique nearest neighbor) are error-corrected, the saturated ideal is just
the prime ideal of a point (so the primary decomposition is superfluously
determined); we show that this ideal can be computed by coloning the original
homogeneous ideal with a power of a certain variable. We then determine the
smallest such power for any linear code.
\end{abstract}
\maketitle

\section{Introduction}

Let $\mathcal C$ be an $[n,k,d]-$linear code with generating matrix (in canonical bases) $$G=\left(\begin{array}{cccc}a_{11}&a_{12}&\cdots&a_{1n}\\ a_{21}&a_{22}&\cdots&a_{2n}\\\vdots&\vdots& &\vdots\\ a_{k1}&a_{k2}&\cdots&a_{kn}\end{array}\right),$$ where $a_{ij}\in\mathbb K$, any field.

By this, one understands that $\mathcal C$ is the image of the injective linear map $$\phi:\mathbb K^k\stackrel{G}\longrightarrow \mathbb K^n.$$ $n$ is the {\em length} of $\mathcal C$, $k$ is the {\em dimension} of $\mathcal C$ and $d$ is the {\em minimum distance (or Hamming distance)}, the smallest number of non-zero entries in a non-zero codeword (i.e.\ non-zero element of $\mathcal C$). For background on linear codes we recommend \cite{hp1}.

Also, for any vector $w\in\mathbb K^n$, the {\em weight} of $w$, denoted $wt(w)$, is the number of non-zero entries in $w$.

\medskip

The most commonly used method for decoding a received word $w\in\mathbb K^n$ is to find the codeword $v\in\mathcal C$ which minimizes $wt(w-v)$ (i.e.\ $v$ is the {\em nearest neighbor} of $w$), and decode $w$ to $\phi^{-1}(v)$. Of course, a $w\notin \mathcal C$ might have more than one nearest neighbor. In this case the nearest neighbor algorithm fails. Fortunately, under certain conditions (see Proposition 2.1 in \cite[Chapter 9]{clo}), error detection and correction are guaranteed to succeed:
\begin{itemize}
  \item any $d-1$ errors in a received vector can be detected, meaning that if there is a $v\in\mathcal C$ with $0<wt(w-v)\leq d-1$, then $w\notin \mathcal C$, and
  \item if $d\geq 2t+1$, any $t$ errors can be corrected, meaning that there is a unique $v\in\mathcal C$ with $wt(w-v)\leq t$.
\end{itemize}

A vector has at most $m$ non-zero entries if and only if all products of $m+1$ distinct entries are zero. This simple result was first exploited in the context of coding theory by De Boer and Pellikaan (\cite{dp}). Furthermore, one can translate the syndrome decoding algorithm, a widely used algorithm based on the method expressed above, into the language of varieties (called {\em syndrome varieties}) and use computational algebraic techniques (such as Gr\"{o}bner bases) to find the errors and the nearest neighbors of a received word: see \cite{dp} (for cyclic codes) and \cite{BuPe} (for a general approach with great applications to error-correcting words received through MDS codes).

We are particularly interested in this approach because of the use of commutative/homological algebra. Even though in the two papers mentioned above the authors end up using Gr\"{o}bner bases for their calculations, fundamental concepts and techniques lie at the foundations of their work (e.g.\ the height of an ideal (\cite{dp}) and the classical Eagon-Northcott complex (\cite{BuPe})). In these notes we take an even more theoretical approach (no Gr\"{o}bner bases analysis), as we want to understand what it means from an abstract point of view to error-correct a word received through any linear code. Furthermore, in order to improve our symbolic computations, and by means of standard coding theory techniques (such as puncturing a code), we end up showing that a certain class of ideals generated by products of linear forms, which define scheme-theoretically the projective codewords of minimum weight, have linear graded minimal free resolutions; the result presented (Theorem \ref{regularity}) is in line with the theme of the landmark paper of Eisenbud and Goto on modules with linear free resolutions, \cite{EiGo}.

Another argument supporting a commutative algebraic approach to linear codes is the fact that the nice homological properties of defining ideals of star configurations can be obtained immediately from the theoretical results concerning MDS codes (see \cite[Remark 2.13]{GeHaMi}).

In conclusion, the benefits of this approach lie in answering questions in commutative algebra and having a broad perspective on linear codes from this direction, rather than in improving/creating better practical algorithms for decoding and error-correcting linear codes, or for computing their minimum distance.

\section{Error-correction via colon ideals}

The basic idea of our strategy to error-correct any received word is the following:

\begin{enumerate}
  \item To the generating matrix of our linear code, augment the received word as a new row. This new matrix is the generating matrix for a new linear code, and under certain conditions (see the two bullets in the Introduction) errors in the transmission become codewords of minimum weight in this new linear code (see Corollary \ref{error_minweight}).
  \item Use techniques from \cite{t}, that consist of saturation of ideals and primary decomposition, to determine these special codewords.
  \item When good words are received (meaning vectors with unique nearest neighbors), both of these techniques are incorporated into one simple operation: colon a certain ideal by a power of a variable (see Lemma \ref{lem_sat}).
\end{enumerate}

Below we explore each of the steps in the strategy.

\subsection{Errors as codewords of minimum weight.} Let $\mathcal C$ be an $[n,k,d]-$linear code with generating matrix $G$ as in the introduction. Suppose that a word $w=(w_1,w_2,\ldots,w_n)\in\mathbb K^n$ is received. Create a new linear code $\mathcal C^w$ with generating matrix $$G^w:=\left(\begin{array}{cccc}a_{11}&a_{12}&\cdots&a_{1n}\\ a_{21}&a_{22}&\cdots&a_{2n}\\\vdots&\vdots& &\vdots\\ a_{k1}&a_{k2}&\cdots&a_{kn}\\w_1&w_2&\cdots&w_n\end{array}\right).$$ Observe that $G^w$ is created from the generator matrix $G$ of $\mathcal C$ by augmenting the extra row $w$; a code with such a generating matrix is called {\em augmented code}.

Let $d_w:=\min\{wt(\epsilon)| \epsilon\in\mathbb K^n \mbox{ with }w-\epsilon\in\mathcal C\}$.

Two codewords are called {\em projectively equivalent} if they differ by multiplication by a non-zero scalar, and such an equivalence class, denoted with square brackets, is called a {\em projective codeword}. For any linear code $\mathcal D$, denote with $\mathbb P\mathcal D(u)$ the set of projective codewords of weight $u$ in $\mathcal D$.

The next result is somewhat folklore in coding theory (it seems that it appears in \cite{Ba}), but for the sake of completeness we present it in the form we will use further in the article, with a complete simple proof.

\begin{thm}\label{error} Let $w\notin \mathcal C$. Then, the nearest neighbors of $w$ in $\mathcal C$ (i.e.\ $v\in\mathcal C$ such that $wt(w-v)$ is minimized) are in one-to-one correspondence with the projective codewords of weight $d_w$ in $\mathcal C^w$ but not in $\mathcal C$.
\end{thm}
\begin{proof} First observe that $w\notin\mathcal C$ is equivalent to $d_w\geq 1$.

Consider the function $$\Phi: \{\mbox{nearest neighbors of }w\mbox{ in }\mathcal C\}\rightarrow \mathbb P\mathcal C^w(d_w) - \mathbb P\mathcal C(d_w),$$ given by $\Phi(v)=[w-v]$.

$\bullet$ {\em $\Phi$ is well-defined:} Let $v$ be a nearest neighbor of $w$ in $\mathcal C$. Then $v$ minimizes $wt(w-v)$, and so $wt(w-v) = d_w$. It is obvious that $w-v\in\mathcal C^w$, as it is a linear combination of the rows of $G^w$. If $[w-v]=[v']$ with $v'\in\mathcal C$, then $w-v=\mu v',$ for some $\mu\in\mathbb K-\{0\}$, and hence $w=v+\mu v'\in\mathcal C$, a contradiction.

$\bullet$ {\em $\Phi$ is injective:} If $\Phi(v_1)=\Phi(v_2)$, then $[w-v_1]=[w-v_2]$. Hence $w-v_1=\mu(w-v_2)$, for some $\mu\neq 0$ in $\mathbb K$. If $\mu=1$, then obviously $v_1=v_2$. Otherwise, we have $$w=\frac{1}{\mu-1}(\mu v_2-v_1) \in \mathcal C,$$ a contradiction.

$\bullet$ {\em $\Phi$ is surjective:} Let $\epsilon\in \mathcal C^w-\mathcal C$ with $wt(\epsilon)=d_w$. We have that $$\epsilon=\lambda w+v,$$ for some $v\in\mathcal C$ and $\lambda\neq 0$ (otherwise, $\epsilon\in \mathcal C$). Since $wt(\frac{1}{\lambda}\epsilon)=wt(\epsilon)=d_w$, then $v':=-\frac{1}{\lambda}v\in\mathcal C$ is a nearest neighbor of $w$ since $wt(w-v')=d_w$, the minimum possible. Obviously $$\Phi(v')=[w-v']=[\frac{1}{\lambda}\epsilon]=[\epsilon],$$ and the proof is complete. \end{proof}

\begin{rem} \label{error_rem}
It should be noted that Theorem \ref{error} can be ``extended'' to the situation in which $w\in \mathcal C$. In this instance $\mathcal C^w=\mathcal C$, and $d_w=0$. Since $w \in \mathcal C$, it is its own nearest neighbor. This corresponds to the only codeword in $\mathcal C=\mathcal C^w$ of weight equal to $d_w=0$, namely the zero vector, which can be written as as $w-w$.
\end{rem}

Theorem \ref{error} is particularly useful when $w\notin\mathcal C$ is a received word such that there exist $v\in\mathcal C$ with $wt(w-v)\leq d-1$. This is the situation in the first bullet of the cited result presented in the Introduction.

\begin{cor}\label{error_minweight} With respect to the notation used previously, if $d_w\leq d-1$, then $\mathcal C^w$ is an $[n,k+1,d_w]-$linear code, and therefore the nearest neighbors of $w$ in $\mathcal C$ (hence the errors) are in one-to-one correspondence with the projective codewords of minimum weight of $C^w$.
\end{cor}
\begin{proof} Since there are no codewords in $\mathcal C$ of weight $d_w$, the set $\mathbb P\mathcal C(d_w)$ is empty.

Furthermore, $w\notin\mathcal C$ assures that the dimension of $\mathcal C^w$ is $k+1$, and the minimality of $d_w$ assures that $\mathcal C^w$ has minimum distance $d_w$. The result then follows from Theorem \ref{error}.
\end{proof}

\begin{rem}\label{find_error} The second bullet of the cited result in the Introduction translates into the following: if $1\leq d_w\leq \lfloor(d-1)/2\rfloor$, then $|\mathbb P\mathcal C^w(d_w)|=1$, meaning that $\mathcal C^w$ has exactly one projective codeword of minimum weight. In the next subsections we will show that it is possible to avoid some of the computational challenges associated with calculating this codeword.

Until then, we can determine immediately from this projective codeword the error in $w$. This projective codeword is $[x]$, where $$x=\underbrace{\lambda_1\cdot r_1(G^w)+\cdots+\lambda_k\cdot r_k(G^w)}_{\in\mathcal C}+\lambda\cdot w\in\mathcal C^w,$$ where $r_i(G^w)$ denotes the $i-$th row of $G^w$, $\lambda_i,\lambda\in \mathbb K$, and $\lambda\neq 0$ (otherwise $x\in\mathcal C$). So $[x]$ can be thought as the projective point $[\lambda_1,\ldots,\lambda_k,\lambda]$ in $\mathbb P^k$, and then the error in $w$ is the (affine) representative of this point in the affine open patch given by taking the last coordinate to be $1$.
\end{rem}

Though the following sits aside from the symbolical/theoretical approach to error-correction of linear codes (the driving force behind these notes), we end this part by mentioning that Corollary \ref{error_minweight} has been considered as a viable alternative to the syndrome-decoding of received words, in the context of attacking code-based cryptographic systems (such as McEliece or Niederreiter public-key systems). A good review can be found in \cite{BeLaPe}, and an extensive references list in \cite{Ca}.

\subsection{Finding projective codewords of minimum weight.} In this subsection we briefly describe the method presented in \cite{t} to obtain information about projective codewords of minimum weight from the commutative algebraic point of view.\footnote{For background on commutative algebra we suggest \cite{clo2} and \cite{sh}.}

Let $\mathcal C$ be an $[n,k,d]-$linear code with generating matrix $G$ of size $k\times n$. To each column $j$ of $G$ we associate a homogeneous linear form $L_j$ in $R:=\mathbb K[x_1,\ldots,x_k]$ with coefficients being the entries in the corresponding column $$L_j=a_{1j}x_1+a_{2j}x_2+\cdots+a_{kj}x_k.$$ Then, create the ideals $$I_s(\mathcal C)=\langle\{L_{j_1}\cdots L_{j_s}\}_{1\leq j_1<\cdots<j_s\leq n}\rangle\subset R.$$ \cite[Theorem 3.1]{t} shows that $d$ is the maximum integer $s$ such that the $\mathbb K-$vector subspace of $R_s$ spanned by the generators of $I_s(\mathcal C)$ has dimension ${{k+s-1}\choose{s}}$.

Concerning projective codewords of minimum weight, by \cite[Lemma 2.2]{t}, $I_{d+1}(\mathcal C)$ has primary decomposition $$I_{d+1}(\mathcal C)={\rm q}_1\cap\cdots\cap {\rm q}_m\cap J,$$ where ${\rm q}_i$ are prime ideals in $R$ each defining a point in $\mathbb P^{k-1}$, and $J\subset R$ with $\sqrt{J}=\langle x_1,\ldots,x_k\rangle$. The homogeneous coordinates of each point $V({\rm q}_i)\in\mathbb P^{k-1}$ give the coefficients in the linear combination of the rows of $G$ that equals a projective codeword of weight $d$.

From this perspective, there are two immediate consequences

\begin{itemize}
  \item The number of projective codewords of minimum weight equals the degree of the ideal $I_{d+1}(\mathcal C)$ (see \cite[Corollary 2.3]{t}), i.e.\ $$m=\deg(I_{d+1}(\mathcal C)).$$
  \item Because the multiplicity of each ${\rm q}_i$ is one, finding ${\rm q}_1\cap\cdots\cap {\rm q}_m$, and hence finding the projective codewords of minimum weight, it is enough to saturate the ideal $I_{d+1}(\mathcal C)$ rather than computing its radical. In general, for an ideal $I\subset R=\mathbb K[x_1,\ldots,x_k]$, the {\em saturation of $I$} is $\sat(I)=\{f\in R|f\in I:\langle x_1,\ldots,x_k\rangle^{n(f)}\mbox{ for some }n(f)\geq 1\}=I:\langle x_1,\ldots,x_k\rangle^{\infty}$.
\end{itemize}

\medskip

An immediate consequence of Theorem \ref{error} is the following recursive method of counting projective codewords of minimum weight. Let $\mathcal C$ be an $[n,k,d]-$linear code with generating matrix $G$. Let $j\in\{1,\ldots,k\}$ and let $\mathcal C_j$ be the linear code with generating matrix $G_j$ obtained by removing row $j$ from $G$. Then, $\mathcal C_j$ has length $n$ and dimension $k-1$. Denote by $d_j$ its minimum distance. Since $\mathcal C_j\subsetneq \mathcal C$, we have $$d_j\geq d.$$ Let $\alpha_s(\mathcal C)$ be the number of projective codewords of $\mathcal C$ of weight $s$, and denote by ${\rm n.n.}(w,\mathcal C)$ the number of nearest neighbors in $\mathcal C$ of a $w\notin\mathcal C$. Note that if $d_j>d$, then $\alpha_d(\mathcal C_j)=0$.

\begin{cor}\label{mincodewords} Let $r_j(G)$ denote the $j-$th row of $G$. With the notations above we have
$$\alpha_d(\mathcal C)-\alpha_d(\mathcal C_j)={\rm n.n.}(r_j(G),\mathcal C_j)=\deg(I_{d+1}(\mathcal C):x_j).$$
\end{cor}
\begin{proof} First, since $G$ is a $k\times n$ matrix of rank $k$, $r_j(G)\notin\mathcal C_j$.

Write $$\alpha_d(\mathcal C) = [\alpha_d(\mathcal C) - \alpha_d(\mathcal C_j)] + \alpha_d(\mathcal C_j).$$ The expression in brackets counts the number of projective codewords of weight $d$ in $\mathcal C$ but not in $\mathcal C_j$.

Setting $w = r_j(G)$ in Theorem \ref{error}, since $\mathcal C = (\mathcal
C_j)^w$ we have $$\alpha_d(\mathcal C) - \alpha_d(\mathcal C_j) = {\rm n.n.}(r_j(G), \mathcal C_j).$$

For the second equality, consider the following classical exact sequence of graded $R-$modules
$$0\longrightarrow \frac{R(-1)}{I_{d+1}(\mathcal C):x_j}\stackrel{\cdot x_j}\longrightarrow \frac{R}{I_{d+1}(\mathcal C)}\longrightarrow \frac{R}{I_{d+1}(\mathcal C)+\langle x_j\rangle}\longrightarrow 0.$$

If $A:=R/\langle x_j\rangle=\mathbb K[x_1,\ldots,\hat{x_j},\ldots,x_k]$, then we have the isomorphism $$\frac{R}{I_{d+1}(\mathcal C)+\langle x_j\rangle}\cong\frac{A}{I_{d+1}(\mathcal C_j)}.$$

All the ideals considered define projective schemes of dimensions $\leq 0$, so the Hilbert polynomial equals the degree of the corresponding ideal. The additivity of Hilbert polynomials under exact sequences proves the claim.
\end{proof}

\subsection{Using colon ideals to error-correct received words.} To find projective codewords of minimum weight, one could solve the ideal $I_{d+1}(\mathcal C)$ using Gr\"{o}bner bases (as \cite{dp} and \cite{BuPe} do), or find a primary decomposition for $\sat(I_{d+1}(\mathcal C))$; both methods are computationally expensive.

We return to the situation of Corollary \ref{error_minweight}: $1\leq d_w\leq d-1$. With regard to saturations, we have the following lemma:

\begin{lem}\label{lem_sat} Consider $I_{d_w+1}(\mathcal C^w)\subset S:=R[T]=\mathbb K[x_1,\ldots,x_k,T]$. If $1\leq d_w\leq d-1$, then there exists a positive integer $u\geq 1$ such that $$\sat(I_{d_w+1}(\mathcal C^w))=I_{d_w+1}(\mathcal C^w): T^u.$$
\end{lem}
\begin{proof} We have $$I_{d_w+1}(\mathcal C^w)=\underbrace{\bar{\rm q}_1\cap\cdots\cap \bar{\rm q}_m}_{\sat(I_{d_w+1}(\mathcal C^w))} \cap\ \bar{J},$$ where $\bar{\rm q}_i$ are prime ideals in $S$ and $\bar{J}\subset S$ with $\sqrt{\bar{J}}=\langle x_1,\ldots,x_k,T\rangle$. Then there exists a positive integer $u\geq 1$ for which $T^u\in \bar{J}$.

$T$, and therefore $T^u$, is a non-zero divisor in $S/\bar{\rm q}_1\cap\cdots\cap\bar{\rm q}_m$, since otherwise one of the points $V(\bar{\rm q}_i)$ would have its last coordinate $0$, meaning that there would be a codeword of $\mathcal C^w$ of weight $\leq d-1$ which is a linear combination of the first $k$ rows of $G_w$ and hence a codeword of $\mathcal C$.

Then $$I_{d_w+1}(\mathcal C^w):T^u=(\sat(I_{d_w+1}(\mathcal C^w)):T^u)\cap (\bar{J}:T^u)=\sat(I_{d_w+1}(\mathcal C^w)),$$ and the proof is complete.
\end{proof}

It is desirable to have an upper bound $v$ for the $u$ above that depends only on $n,k,d$ and/or $d_w$, because then $I_{d_w+1}(\mathcal C^w):T^{v}=I_{d_w+1}(\mathcal C^w):T^{u}$. Then, one could avoid using a recursive method to find the saturation. Finding such an upper bound is equivalent to finding an upper-bound for the index of saturation and, consequently, to finding an upper bound for the Castelnuovo-Mumford regularity. It is well known that the regularity provides an upper bound for the complexity of Gr\"{o}bner basis algorithms that solve an ideal, and this is rather difficult to present. As the next example shows, we believe that this colon ideal method is quite easy to understand and implement even for non-experts.

\begin{exm} \label{sat} Let us consider the linear code over $\mathbb F_2$ with generating matrix $$G=\left(\begin{array}{cccccc} 1&0&0&1&1&0\\ 0&1&0&1&0&1\\ 0&0&1&0&1&1 \end{array}\right).$$ Suppose the word $w=(0,1,1,1,0,0)$ is received.

We have $$G^w=\left(\begin{array}{cccccc} 1&0&0&1&1&0\\ 0&1&0&1&0&1\\ 0&0&1&0&1&1\\0&1&1&1&0&0 \end{array}\right).$$ In $\mathbb P^3$, we have the six linear forms corresponding to the columns of $G^w$:
\begin{eqnarray}
L_1&:=& x\nonumber\\
L_2&:=& y+T\nonumber\\
L_3&:=& z+T\nonumber\\
L_4&:=& x+y+T\nonumber\\
L_5&:=& x+z\nonumber\\
L_6&:=& y+z.\nonumber
\end{eqnarray}

For the next computations we use Macaulay2 (\cite{GrSt}).

Create $$I_2(\mathcal C^w)=\langle \{L_iL_j\}_{1\leq i<j\leq 6}\rangle$$ and calculate its height: $ht(I_2(\mathcal C^w))=3$. Because $3<4=ht(I_1(\mathcal C^w))$, by \cite{dp} we obtain $$d_w=1=(d-1)/2.$$

Next we colon $I_{d_w+1}(\mathcal C^w)$ by successive powers of $T$ until we obtain the ideal of a point. We have $$I_2(\mathcal C^w):T=\langle x,y+T,z+T\rangle.$$ This is the ideal of the projective point $P_w:=[0,-\lambda,-\lambda,\lambda]\in\mathbb P^3, \lambda\neq 0$.

The projective codeword of minimum weight is $$0\cdot r_1(G^w)+ (-\lambda)\cdot r_2(G^w)+(-\lambda)\cdot r_3(G^w)+\lambda\cdot r_4(G^w)=(0,0,0,0,-\lambda,-2\lambda).$$ Since this is over $\mathbb F_2$, we obtain the precise error to be $(0,0,0,0,1,0)$.
\end{exm}

\section{Regularity of some ideals generated by products of linear forms and error-correction of good words received}

It is known that for any $[n,k,d]-$linear code $\mathcal C$, if $1\leq i\leq d$, then $I_i(\mathcal C)=\langle x_1,\ldots,x_k\rangle ^i$ (see \cite[Theorem 3.1]{t}), which leads naturally to the question of the the structure of $I_i(\mathcal C)$ for $i\geq d+1$. Even when $i=d+1$, in general the embedded component of $I_{d+1}(\mathcal C)$ (which is the defining ideal of the scheme of projective codewords of minimum weight) is not known or well understood, leading to difficulties in finding its non-trivial properties.

When $\mathcal C$ is MDS, everything is known about $I_i(\mathcal C),i\geq d+1$ (see \cite[Proposition 2.9]{GeHaMi}, or \cite[Remark 36]{BuPe}, or the second part of the proof of \cite[Proposition 2.1]{t2}), simply because the Eagon-Northcott complex becomes a free resolution.

\medskip

The Castelnuovo-Mumford regularity, or simply the {\em regularity}, of an ideal $I\subset S:=\mathbb K[x_0,\ldots,x_n]$, denoted ${\rm reg}(I)$, is one of the most important homological invariants in commutative algebra; as mentioned previously it can provide an upper bound on the complexity of the Gr\"{o}bner basis algorithms that solve the ideal $I$. If $$0\rightarrow \oplus_{i=1}^{n_t} S(-b_{i,t})\rightarrow \oplus_{i=1}^{n_{t-1}} S(-b_{i,t-1})\rightarrow\cdots\rightarrow \oplus_{i=1}^{n_0} S(-b_{i,0})\rightarrow I\rightarrow 0$$ is a graded minimal free resolution of $I$, then $${\rm reg}(I)=\max\{b_{i,j}-j: 0\leq j\leq t, 1\leq i\leq n_j\}.$$

From the definition of $\sat(I)$, the saturation of $I$ with respect to the maximal ideal, one usually defines a number called the {\em saturation index} (or index of saturation) of $I$, denoted here ${\rm s.ind.}(I)$, which is the smallest integer $\delta$ such that $(I)_m=(\sat(I))_m$ (the degree $m$ pieces) for all $m\geq\delta$.

The connection between these two numbers is $${\rm reg}(I)=\max\{{\rm s.ind.}(I),{\rm reg}(\sat(I))\}.$$ See \cite{BeGi} for more details.

\medskip

In regard to the regularity of $I_{d+1}(\mathcal C)$, when $\mathcal C$ is any linear code with $\alpha_d(\mathcal C)=1$ (i.e.\ $\mathcal C$ has only one projective codeword of minimum weight), after doing a fair amount of examples with \cite{GrSt} we arrived to the same conclusion as one of the referees: ${\rm reg}(I_{d+1}(\mathcal C))=d+1$, which, by \cite[Theorem 1.2(2)]{EiGo}, is equivalent to $I_{d+1}(\mathcal C)$ having linear graded free resolution. Below we give a proof of this interesting problem, for a particular case that fits very well within our error-correction discussion.

First, if $\deg(I_{d_w+1}(\mathcal C^w))>1$, then $w$ has at least two nearest neighbors in $\mathcal C$, and in practice the received word is requested to be sent again. So we will only consider the case when $\deg(I_{d_w+1}(\mathcal C^w))=1$, meaning that $\sat(I_{d_w+1}(\mathcal C^w))$ consists of just one prime ideal, and finding its primary decomposition becomes superfluous. So for the remainder of these notes we assume the primary decomposition $$I_{d_w+1}(\mathcal C^w)=\bar{\rm q}\cap\bar{J}\subset S:=\mathbb K[x_1,\ldots,x_k,T],$$ where $\bar{\rm q}\subset S$ is a prime ideal of codimension (i.e.\ height) $k$ generated by linear forms and $\sqrt{\bar{J}}=\langle x_1,\ldots,x_k,T\rangle$. Note that if $1\leq d_w\leq \lfloor(d-1)/2\rfloor$ this is always the case.

Since $d_w\geq 1$, and $I_{d_w+1}(\mathcal C^w)$ is generated in degree $d_w+1$, then $${\rm reg}(I_{d_w+1}(\mathcal C^w))\geq d_w+1>1={\rm reg}(\underbrace{\sat(I_{d_w+1}(\mathcal C^w))}_{\bar{\rm q}},$$ giving that the regularity and the saturation index of $I_{d_w+1}(\mathcal C^w)$ must coincide.

\medskip

Furthermore, we may assume that the received word $w$ has weight equal to
$d_w$, because we can write $w = v+\epsilon$ with $v \in \mathcal C$ and
$\epsilon \in \mathbb K^n$ such that $wt(\epsilon) = d_w$. Let $G$ be the
generating matrix of $\mathcal C$. Then $v \in \mathcal C$ is a linear
combination of the rows of $G$. Reducing the last row of $G^w$ by the
coefficients of this linear combination, we obtain $G^{\epsilon}$, which is
also a generating matrix for $\mathcal C^w = \mathcal
C^{\epsilon}$. Therefore, we may assume $w = \epsilon$.

\begin{thm}\label{regularity} Let $\mathcal C$ be any $[n,k,d]-$linear code, $d\geq 3$, and let $w\in\mathbb K^n$ be of weight $1\leq m\leq \lfloor(d-1)/2\rfloor$. Then $${\rm reg}(I_{m+1}(\mathcal C^w))= m+1.$$
\end{thm}
\begin{proof} After an appropriate permutation of the columns of $G$ and consequently of $G^w$, we may assume that $w$ has the canonical form $$w=(0,\ldots,0,a_{n-m+1},a_{n-m+2},\ldots,a_n),\mbox{ where }a_i\neq 0,$$ and $$I_{m+1}(\mathcal C^w)=\underbrace{\langle x_1,\ldots,x_k\rangle}_{\bar{\rm q}=\sat(I_{m+1}(\mathcal C^w))}\cap\ \bar{J}\subset S:=\mathbb K[x_1,\ldots,x_k,T].$$

Also, we denote with $L_i\in S,i=1,\ldots,n$ the linear forms dual to the columns of $G^w$, and with $\ell_i\in R:=\mathbb K[x_1,\ldots,x_k], i=1,\ldots, n$ the linear forms dual to the columns of $G$. We have
\begin{eqnarray}
L_i&=&\ell_i,\mbox{ for }i=1,\ldots,n-m\nonumber\\
L_j&=&\ell_j+a_jT,\mbox{ for }j=n-m+1,\ldots,n.\nonumber
\end{eqnarray}

\noindent \underline{Claim}: $\bar{\rm q}^{m+1}\subset I_{m+1}(\mathcal C^w)$.

\medskip

\noindent \underline{Proof of Claim}. Let $\mathcal C(i)$ be the puncturing of
$\mathcal C$ at the last $i$ columns of $G$.\footnote{We often make use of
  this technique of {\em puncturing} a code. For more details see \cite{hp1},
  page 465.} Then $\mathcal C(i)$ is an $[n-i,k_i,\delta_i]-$ linear
code. Since $d\geq 3 > 1$, then $k_1=k$. Note that $k_{i+1} = k_i$ or $k_i -
1$, since the dimension of the row space of a matrix can only be changed by
$1$ as a result of deleting a column. Suppose that $k_u=k$, but $k_{u+1}=k-1$,
for some $u\in\{1,\ldots,m-1\}$. Then,
$\delta_u=1$. But $$\delta_j-\delta_{j+1}\leq 1,\mbox{ for any
}j=0,\ldots,u-1, \delta_0:=d.$$ So $d\leq u+1\leq m$, contradicting the
hypotheses. The conclusion is that $\mathcal C(m)$ is an
$[n-m,k,\delta_m]-$linear code with $\delta_m\geq d-m$. Since $d-m\geq m+1$,
by \cite[Theorem 3.1]{t}, we have $$I_{m+1}(\mathcal C(m))=\langle
x_1,\ldots,x_k\rangle^{m+1}\subset R:=\mathbb K[x_1,\ldots,x_k].$$ Lifting up
to $S=R[T]$, since $L_i=\ell_i, i=1,\ldots,n-m$, we obtain the desired
inclusion $\bar{\rm q}^{m+1}\subset I_{m+1}(\mathcal C^w)$.

\vskip .1in

We will prove by induction on $n-k\geq 2$ that $$\bar{\rm q}^{m-i+1}T^i\subset I_{m+1}(\mathcal C^w), i=0,\ldots,m.$$ Given this, with $I_{m+1}(\mathcal C^w)\subset \bar{\rm q}$, we obtain $(I_{m+1}(\mathcal C^w))_{m+1}=(\bar{\rm q})_{m+1}$. So $$m+1\leq{\rm reg}(I_{m+1}(\mathcal C^w))={\rm s.ind.}(I_{m+1}(\mathcal C^w))\leq m+1,$$ and hence the proof of the theorem.

\vskip .1in

\noindent \underline{Special case}: When $m=1$, we have $\bar{\rm q}^2\subset I_2(\mathcal C^w)$ and $\bar{\rm q}\cdot T\subset I_2(\mathcal C^w)$.

We have $L_1=\ell_1,\ldots,L_{n-1}=\ell_{n-1},L_n=T+\ell_n$, and therefore $$I_2(C^w)=\langle \{\ell_i\ell_j\}_{1\leq i<j\leq n-1},\ell_1T+\ell_1\ell_n,\ldots,\ell_{n-1}T+\ell_{n-1}\ell_n\rangle.$$ With the notations in the Claim above, we have $\langle \{\ell_i\ell_j\}_{1\leq i<j\leq n-1}\rangle=I_2(\mathcal C(1))$. Since $d\geq 3$, then $k_1=k$, and $\delta_1=d$ or $d-1$, hence $\delta_1\geq 2$. So, by \cite[Theorem 3.1]{t}, $I_2(\mathcal C(1))=\langle x_1,\ldots,x_k\rangle^2$, giving that $$I_2(C^w)=\langle \{\ell_i\ell_j\}_{1\leq i<j\leq n-1},\ell_1T,\ldots,\ell_{n-1}T\rangle.$$ Therefore $$\bar{\rm q}=\langle x_1,\ldots,x_k\rangle=\langle \ell_1,\ldots,\ell_{n-1}\rangle\subseteq I_2(C^w):T.$$ Together with the Claim, the special case is shown.

\vskip .1in

We now move to an inductive proof of Theorem \ref{regularity}.

\noindent {\bf Case $n-k=2$; base case.} Then, from Singleton bound $d\leq n-k+1$, we necessarily have $d=3$ and $m=1$. The special case proves the base case.

\vskip .1in

\noindent {\bf Case $n-k\geq 3$; induction step.} From the special case above, we may assume $m\geq 2$.

Let $\mathcal C':=\mathcal C(1)=\mathcal C-\{\ell_n\}$ be the puncturing of $\mathcal C$ at the last column of $G$. Since $d\geq 3$, the dimension of $\mathcal C'$ is $k':=k_1=k$, and the minimum distance is $d':=\delta_1=d$ or $d-1$. The length of $\mathcal C'$ is $n'=n-1$.

Let $w'=(0,\ldots,0,a_{n-m+1},\ldots,a_{n-1})\in \mathbb K^{n-1}$, obtained from $w$ by removing the last entry. Then, by keeping with the notation used throughout this paper, we have $$(\mathcal C')^{w'}=\mathcal C^w-\{L_n\}.$$

If $k_{w'}$ and $d_{w'}$ are the dimension and the minimum distance of this new linear code, respectively, since $d_w=m\geq 2$, then $k_{w'}=k'+1=k+1$. As the weight of $w'\in(\mathcal C')^{w'}$ is $m-1$, we must have $d_{w'}=m-1$.

Let $m':=m-1$. Then $m'\leq \lfloor(d-1)/2\rfloor-1\leq
\lfloor(d-2)/2\rfloor\leq\lfloor(d'-1)/2\rfloor$. Also, by the construction of
$w'$ from $w$, we have $$\sat(I_{m'+1}((\mathcal C')^{w'})=\langle
x_1,\ldots,x_k\rangle=\bar{\rm q},$$ and by the inductive step, $\bar{\rm q}^{m'-i+1}T^i\subset I_{m'+1}((\mathcal C')^{w'}), i=0,\ldots,m'$.

\medskip

It is clear that $$I_{m+1}(\mathcal C^w)=L_n\cdot I_{m'+1}((\mathcal C')^{w'})+I_{m+1}((\mathcal C')^{w'}).$$ We write $L_n=a_nT+\ell_n$, and let $L_{i_1}\cdots L_{i_m}\in I_{m}((\mathcal C')^{w'})$ be arbitrary. Suppose $i_1,\ldots,i_s\in\{1,\ldots,n-m\}$, for some $m\geq s\geq 1$.

Let us consider once again the linear code $\mathcal C(m)$ we have seen in the Claim. Let $$\mathcal C(m;i_1,\ldots,i_j)=\mathcal C(m)-\{\ell_{i_1},\ldots,\ell_{i_j}\}, j=1,\ldots,s.$$ Denote with $k_{m,j}$ and $\delta_{m,j}$ the dimension and, respectively, the minimum distance of $\mathcal C(m;i_1,\ldots,i_j)$.

We repeat the argument in the proof of Claim to show that $k_{m,s}=k$. Suppose to the contrary that for some $v\in\{1,\ldots,s-1\}$, we have $k_{m,v}=k$ and $k_{m,v+1}=k-1$. Then, $\delta_{m,v}=1$. Since $\delta_{m,j}-\delta_{m,j+1}\leq 1, j=0,\ldots,v-1,$ where $\delta_{m,0}:=\delta_m$, adding all these inequalities, we obtain $\delta_m\leq v+1$. Since $v\leq s-1$, and $s\leq m$, together with $\delta_m\geq d-m$, we obtain $$d-m\leq m,$$ which contradicts the hypotheses of the theorem.

As $n-k\geq d-1\geq 2m$, there exist indices $j_1,\ldots,j_k\in \{1,\ldots,n-m\}-\{i_1,\ldots,i_{m}\}$, such that $\ell_{j_1}=L_{j_1},\ldots,\ell_{j_k}=L_{j_k}$ are linearly independent (from $k_{m,s}=k$). So $\ell_n$ can be written as a linear combination of $L_{j_1},\ldots,L_{j_k}$, giving that $$\ell_nL_{i_1}\cdots L_{i_m}\in I_{m+1}((\mathcal C')^{w'}),$$ and so $$I_{m'+1}((\mathcal C')^{w'})\subseteq I_{m+1}(\mathcal C^w):T.$$

With the induction step and $m'=m-1$, we have $\bar{\rm q}^{m-i}T^{i+1}\subset I_{m+1}((\mathcal C)^w), i=0,\ldots,m-1$. Denoting $i+1=:j$, we have $$\bar{\rm q}^{m-j+1}T^j\subset I_{m+1}((\mathcal C)^w), j=1,\ldots,m,$$ and with the Claim at the beginning, the result is shown.
\end{proof}

\begin{cor} Let $\mathcal C$ be any $[n,k,d]-$linear code, $d\geq 3$, and let $w\in\mathbb K^n$ be such that $1\leq d_w\leq \lfloor(d-1)/2\rfloor$. Then, $$\sat(I_{d_w+1}(\mathcal C^w))=I_{d_w+1}(\mathcal C^w): T^{d_w}.$$
\end{cor}
\begin{proof} Denote $m=d_w$, and as before, we may assume $$I_{m+1}(\mathcal C^w)=\underbrace{\langle x_1,\ldots,x_k\rangle}_{\bar{\rm q}=\sat(I_{m+1}(\mathcal C^w))}\cap\ \bar{J}\subset S:=\mathbb K[x_1,\ldots,x_k,T].$$

From the proof of Theorem \ref{regularity} we have $\bar{\rm q}\cdot T^m\subset I_{m+1}(\mathcal C^w)$, and therefore $\bar{\rm q}\subset I_{m+1}(\mathcal C^w):T^m$. The colon ideal is included in $\bar{\rm q}$, by the proof of Lemma \ref{lem_sat}, and hence we have equality throughout.
\end{proof}

\vskip .1in

\noindent{\bf Acknowledgement} The first author is supported by the Brian and Gayle Hill Fellowship awarded by the College of Science at University of Idaho.

We are grateful to the two anonymous referees for important suggestions and corrections.

\renewcommand{\baselinestretch}{1.0}
\small\normalsize 

\bibliographystyle{amsalpha}

\end{document}